\documentclass{amsart}
\usepackage[margin=1in]{geometry}
\usepackage{amsmath, amsfonts, amssymb, amsthm}
\usepackage{hyperref}
\usepackage{color}
\usepackage{dsfont}
\usepackage{wasysym}
\usepackage{graphicx}

\newtheorem{theorem}{Theorem}[section]
\newtheorem{proposition}[theorem]{Proposition}

\newtheorem{corollary}[theorem]{Corollary}

\newtheorem{lemma}[theorem]{Lemma}

\newtheorem{preremark}[theorem]{Remark}

\numberwithin{equation}{section}

\newcommand{\R}{\mathbb R}

\newcommand{\dd} {\; \mathrm{d}}

\newcommand{\Sp}{\mathds{S}}

\title{Entropy dissipation estimates for the Boltzmann equation without cut-off}
\date{\today}

\author{Jamil Chaker}
\address{(\textnormal{J. Chaker}), Fakult\"at f\"ur Mathematik, Universit\"at Bielefeld, 33615 Bielefeld, Germany}
\email{jchaker@math.uni-bielefeld.de}

\author{Luis Silvestre}
\address{(\textnormal{L. Silvestre}), Mathematics Department, University of Chicago, Chicago, Illinois 60637, USA}
\email{luis@math.uchicago.edu}

\thanks{Jamil Chaker gratefully acknowledges funding by the Deutsche Forschungsgemeinschaft (SFB 1283/2 2021 - 317210226). Luis Silvestre is supported by NSF grants 2054888 and 1764285.}

\begin{document}
\begin{abstract}
We prove a bound on the entropy dissipation for the Boltzmann collision operator from below by a weighted $L^p$-Norm. The estimate holds for a wide range of potentials including soft potentials as well as very soft potentials. 
As an application, we study weak solutions to the spatially homogeneous Boltzmann equation and prove a weighted 
$L^1_t(L^p_v)$ estimate.
\end{abstract}
\maketitle

We study the entropy dissipation for the Boltzmann collision operator without cutoff for a wide range of power law potentials.
The main result of the present work gives a bound on the entropy dissipation from below by a weighted Lebesgue-Norm. 

The Boltzmann equation is a nonlinear integro-differential equation which describes the dynamics of a diluted gas. 
It is of the form
\begin{equation}\label{eq:inhomogboltzmann}
 \partial_t f(t,x,v)  + v\cdot\nabla_x f(t,x,v) = Q(f,f)(t,x,v), \quad f(0,x,v) = f_0(x,v),
 \end{equation}
where $t\geq 0$ and $x,v\in\R^d$ for $d\geq 2$. 
The solution $f$ describes the density of particles at time $t\geq 0$ with position $x\in\R^d$ having velocity $v\in\R^d$. 
While the left-hand side of \eqref{eq:inhomogboltzmann} describes the transport of particles, the right-hand side takes interactions between particles into account. 

In the special case of spacial homogeneity, the Boltzmann equation simplifies to
\begin{equation}\label{eq:homogboltzmann}
 \partial_t f(t,v)  = Q(f,f)(t,v), \quad f(0,v) = f_0(v).
 \end{equation}

The operator $Q$ on the right-hand side of the Boltzmann equation denotes the so-called Boltzmann collision operator, which acts on the function $f(t,x,\cdot)$ for fixed values of $t,x$ and is given by\footnote{Throughout the paper, we will drop the dependence on $t$ and $x$ of a function, whenever they are not directly involved.}
\[ Q(g,f)(v) = \int_{\R^{d}}\int_{\mathds{S}^{d-1}} (g(v_{\ast}')f(v') - g(v_{\ast})f(v))B(|v-v_{\ast}|,\cos\Theta)\, \dd\sigma \dd v_{\ast},  \]
where for $v,v_{\ast}\in\R^d$, $\sigma\in\mathds{S}^{d-1}$
\begin{center}
\begin{tabular}{ l l }
$\displaystyle v' = \frac{v+v_{\ast}}{2} + \frac{|v_{\ast}-v|}{2}\sigma$, \hspace*{2.5cm}  & $\displaystyle\cos\Theta= \sigma\cdot\frac{v_{\ast}-v}{|v_{\ast}-v|}$, \\
$\displaystyle v_{\ast}'= \frac{v+v_{\ast}}{2} - \frac{|v_{\ast}-v|}{2}\sigma.$ & \, 
\end{tabular}
\end{center}
The type of interactions of the particles is determined by the so-called collision kernel $B(|v-v_{\ast}|,\cos\Theta)$. A classical assumption on $B$ is the integrability of the collision kernel, referred to as Grad's cutoff assumption. In this paper, we study collision kernels that do not satisfy Grad's cutoff assumption.
To be more precise, we consider collision kernels of the form $B(|v-v_{\ast}|,\cos\Theta) = \Phi(|v-v_{\ast}|)b(\cos\Theta)$. 
We assume
\begin{equation}\label{def:Phi}
\Phi(|v-v_{\ast}|) = c_{\Phi}|v-v_{\ast}|^{\gamma}
\end{equation} 
for some $c_{\Phi}>0$ and $b$ satisfies
\begin{equation}\label{def:b}
 c_b^{-1} \Theta^{-1-2s} \leq \sin(\Theta)^{d-2}b(\cos\Theta) \leq c_b \Theta^{-1-2s} \quad \text{for all } \Theta\in\left(0,\frac{\pi}{2}\right]
 \end{equation}
for some $c_b>0$, where $\gamma>-d$ and $s\in(0,1)$. The case $\gamma \geq 0$ is referred to as the hard potential case and $\gamma < 0$ as soft potential case. In particular, the sub-case $\gamma+2s<0$ is known as the very-soft potential case. 

Many questions about the regularity of solutions are open in the very soft potential range. For a recent review and open problems, see \cite{ImbSilReg20}. In that case, the reaction term in the collision operator (we write it $Q_2$ in \eqref{def:Q1Q2}) is more singular. It is difficult to control it with the diffusion part of the operator. In particular, there is no known method that leads to $L^{\infty}$ estimates in the very soft potential case, even for space homogeneous solutions. A similar difficulty arises for the Landau equation in the very soft potential range, and in particular for Coulomb potentials. Our main objective in this paper is to derive an entropy dissipation estimate that applies in the very soft potential range, similar to the well known result by L. Desvillettes \cite{DesvillettesColoumb} for the Landau equation with Coulomb potentials. Our main estimate is in a weighted $L^p$ space, with hopefully sharp asymptotics for large velocities.

The entropy dissipation for the Boltzmann collision operator is given by the following formula
\begin{equation}\label{eq:genentropydiss}
\begin{aligned}
D(f) &=-\left\langle Q(f,f),\ln f\right\rangle_{L^2(\R^d)}\\
& = \int_{\R^d}\int_{\R^d}\int_{\Sp^{d-1}} f(v_{\ast})f(v) \left[\ln(f(v))-\ln(f(v'))\right] B(|v-v_{\ast}|,\cos\Theta)\dd\sigma \dd v_{\ast} \dd v.
\end{aligned}
\end{equation}

The expression for $D(f)$ applies to $f$ as a function of $v$, for each \emph{frozen} values of $t$ and $x$. The following entropy dissipation formula applies to solutions of the Boltzmann equation:
\begin{equation} \label{eqn:dissipationintegratedintime}
\partial_t \iint_{\R^d \times \R^d} f \log f \dd v \dd x = -\int_{\R^d} D(f) \dd x.
\end{equation}
In the case of space-homogeneous solutions, the formula is simpler (and more powerful) since it does not involve integration with respect to $x$.

The expression for $D(f)$ is nonnegative. Due to the difficulty to obtain any other coercive quantities associated to the Boltzmann equation, it is interesting to study lower bounds for $D(f)$ that lead to a priori estimates in standard function spaces.

With the aim of sudying estimates for $D(f)$, we consider a nonnegative function $f = f(v) : \R^d \to [0,\infty)$. There is no point in keeping track of the dependence of $f$ with respect to $t$ and $x$ since $D(f)$ applies to $f$ as a function of $v$ only. Our estimate will depend on the mass, energy and entropy of $f$. To be more precise, it depends on upper bounds for the energy and entropy, and upper and lower bounds for the mass of $f$ of the following form:
\begin{align}
& 0 < m_0 \leq \int_{\R^d} f(v) \dd v \leq M_0, \label{ass:mass} \\
& \int_{\R^d} f(v) |v|^2 \dd v \leq E_0, \label{ass:energy} \\
& \int_{\R^d} f(v) \log f(v) \dd v \leq H_0. \label{ass:entropy}
\end{align}
For the spatially homogeneous Boltzmann equation, due to the conservation of mass and energy, and the monotonicity of the entropy, it suffices that these inequalities hold initially for them to hold for positive time. For the space-inhomogeneous Boltzmann equation, the estimates in this paper would apply provided that the inequalities above hold for every value of $t$ and $x$.

Before, we formulate the main result of the present paper, let us recall the weighted Lebesgue Norm. For $p\geq 1$ and $\ell\in\R$, we define
\[ \|g\|_{L^p_{\ell}} =  \left(\int_{\R^d} \langle v \rangle^{\ell p}|g(v)|^p\, \dd v \right)^{1/p}, \]
where $\langle v \rangle = (1+|v|^2)^{1/2}$.
 The main result of the present paper is the following entropy dissipation estimate:
\begin{theorem}\label{thm:entropydissipation}
Let $-d<\gamma \leq 2$ and $s\in(0,1)$. Let $f$ be a non-negative function satisfying \eqref{ass:mass}, \eqref{ass:energy} and \eqref{ass:entropy}. 
Assuming $\gamma \leq 0$, there is a finite constant $c>0$, depending on $d$, $b$, $\gamma$, $s$ and the macroscopic bounds $m_0,M_0,E_0$ and $H_0$, and $C>0$ depending on $d$, $b$, $\gamma$ and $s$ only, such that
\begin{equation}\label{eq:mainresult}
  D(f) \geq c\|f\|_{L^p_{-q}} - C M_0^2,
\end{equation} 
where $1/p = 1-2s/d$ and $q=2s/d - \gamma-2s$.

When $\gamma > 0$, a similar estimate follows but depending on a higher moment instead of $M_0$.
\begin{equation}\label{eq:mainresult-hardpotential}
  D(f) \geq c\|f\|_{L^p_{-q}} - C \left( \int_{\R^d} \langle v \rangle^\gamma f \dd v \right)^2.
\end{equation} 
\end{theorem}
The proof of \autoref{thm:entropydissipation} takes advantage of the simple idea of using the non-negativity of the integrand in the entropy dissipation and replace the kinetic factor $\Phi$ by a smaller bounded function $\psi$ without the singularity on $v=v_{\ast}$ for $\gamma<0$. The estimate in a weighted $L^p$ space, with a precise exponent, follows from an explicit formula for integral quadratic forms. This estimate, given in Proposition \ref{prop:quadratictoLebesgue}, is one of the main novelties of this paper.

As a corollary of \autoref{thm:entropydissipation}, we see that H-solutions to the Cauchy problem \eqref{eq:homogboltzmann} are in a weighted Lebesgue space, that is 
\[ f\in L^{1}\left([0,T], L^p_{-q}(\R^d)\right),\]
 where $p$ and $q$ are as in \autoref{thm:entropydissipation}. This implies in particular that H-solutions to \eqref{eq:homogboltzmann} are weak solutions in the usual sense. For details, see Section \ref{sec:weaksol}.

The Landau equation can be derived as the grazing collision limit of the Boltzmann equation, see e.g. \cite{desvillettes1992asymptotics, goudon1997boltzmann, Villaninewclass, AlexandreVillani2002boltzmann, alexandre2004landau} and the references therein. Precisely, the Boltzmann collision operator $Q(f,f)$, properly normalized, converges to the Landau operator as $s \to 1$. In \cite{DesvillettesColoumb}, Desvillettes proves an entropy dissipation estimate for the Landau equation with Coulomb interaction and presents applications to weak solutions. The estimate we obtain in \autoref{thm:entropydissipation} is an analogous result but for the Boltzmann collision operator. If we take $d=3$ and $\gamma=-3$, and take $s \to 1$ in Theorem \ref{thm:entropydissipation}, we see that $p\to 3$ and $q\to 5/3$. While, the exponent $p$ coincides with the one in \cite{DesvillettesColoumb}, our exponent in the weight $q$ is improved ($L^3_{-5/3}$ as opposed to $L^3_{-3}$), suggesting that the weight exponent in \cite{DesvillettesColoumb} may not be optimal. It is worth noting that the proof given in \cite{DesvillettesColoumb} cannot be applied to the Boltzmann collision operator. The method in this paper is related to a simpler proof presented in \cite{golse2019partial}.

The entropy dissipation is an important quantity in the analysis of the Boltzmann equation. It has various applications such as the construction of renormalized solutions to the Cauchy problem for the Boltzmann equation, see \cite{DiPernaLionsCauchy}, or the introduction of H-solutions for the Boltzmann equation and Landau equations, see \cite{Villaninewclass}. There are various lower bounds for the entropy dissipation in the literature. However, they have serious limitations when $\gamma<0$ that we seek to overcome with the present result.

One of the first and best known entropy dissipation estimates appeared in \cite{AlexDesvVilWenn}. Their analysis applies when $\Phi$ is bounded below which, strictly speaking, is only the case when $\gamma = 0$. For $\gamma < 0$, some further analysis in \cite{AlexDesvVilWenn} leads to \emph{local} estimates restricted to bounded values of $v$. For other works on entropy dissipation estimates and their applications, see for instance \cite{VillaniReg, AlexandreVillani2002boltzmann, Des03Fourier, desvillettes2005trend, mouhot2006explicit} and the references therein.  
In \cite{Gressmann-Strain-Global2010, GressmanStrainGlobal}, Gressmann and Strain introduce a metric which captures the anisotropic structure of the Boltzmann operator. Using this metric and the associated spaces, the same authors obtain entropy dissipation estimates in \cite{Gressmann-Strain-2011} with sharp asymptotics for large velocities. The estimates in \cite{Gressmann-Strain-2011} depend on a quantity (that the authors call $C_g$) that is only controlled by moments of $f$ when $\gamma \geq 0$ (the hard potentials case). 

In addition of our main result, we present a refinement of the entropy dissipation result of \cite[Theorem 3]{Gressmann-Strain-2011} so that it applies to the soft potential range without resorting to higher integrability assumptions on $f$. We recall the anisotropic fractional Sobolev norm of Gressman and Strain:
\begin{equation}\label{def:weightedanisoSobolev}
 |f|_{{\dot{N}^{s,\gamma}}}^2 = \int_{\R^d}\int_{\R^d} \frac{(f(v')-f(v))^2}{d_{GS}(v,v')^{d+2s}} \left( \langle v \rangle \langle v' \rangle \right)^{(\gamma+2s+1)/2} \mathds{1}_{\{d_{GS}(v,v')\leq 1\}} \dd v' \dd v,
\end{equation}
where  
\begin{equation}\label{eq:GSdist}
 d_{GS}(v,v') = \sqrt{|v-v'|^2+\frac{1}{4}(|v|^2-|v'|^2)^2}
 \end{equation}
measures the distance in the lifted paraboloid $\{v\in\R^{d+1}\colon v_{d+1} = \frac12|(v_1,\dots,v_d)|^2 \}$.

We derive the following entropy dissipation estimate for soft potentials.

\begin{proposition}\label{prop:entropydissipation}
Let $-d<\gamma \leq 0$ and $s\in(0,1)$. Let $f$ be a non-negative function satisfying \eqref{ass:mass}, \eqref{ass:energy} and \eqref{ass:entropy}.
There is a constant $c>0$ depending on $d$, $s$, $\gamma$ and the macroscopic bounds $m_0,M_0,E_0$ and $H_0$, and a constant $C$ depending only on $d$, $s$ and $\gamma$ only, such that
\[ D(f) \geq c |\sqrt{f}|_{\dot{N}^{s,\gamma}}^2 - C M_0^2. \]
\end{proposition}

The novelty of Proposition \ref{prop:entropydissipation} compared with \cite[Theorem 3]{Gressmann-Strain-2011} is that our negative error term is in terms of the mass of $f$ only. The result in the cited paper has a negative error term depending on higher integrability assumptions on the function. Roughly, they require $f \ast |v|^\gamma$ to be locally bounded in \cite[Assumption U]{Gressmann-Strain-2011}. There is no apparent upper bound in terms of the hydrodynamic quantities for their parameter $C_g$ when $\gamma<0$. Proposition \ref{prop:entropydissipation} implies that weak solutions to the space-homogeneous Boltzmann equation belong to $L^1([0,T], N^{s,\gamma})$, which was not available from earlier results in the literature.

While it is conceivable that one could potentially derive our main result in Theorem \ref{thm:entropydissipation} from Proposition \ref{prop:entropydissipation} combined with some sharp form of a weighted fractional Sobolev inequality (not readily available in the literature), we chose to prove Theorem \ref{thm:entropydissipation} directly, and then present an independent proof of Proposition \ref{prop:entropydissipation}. The direct proof of Theorem \ref{thm:entropydissipation} is relatively short and elegant. So, we think it is worth presenting Theorem \ref{thm:entropydissipation} as an independent result.

\subsection*{Notation}We write $a\lesssim b$ if there is a universal constant $c>0$ such that $a\leq cb$. The notation $a\gtrsim b$ means that $b\lesssim a$ and $a \approx b$ that $a\lesssim b$ and $a\gtrsim b$.

\subsection{Applications to spatially-homogeneous weak solutions}\label{sec:weaksol}

Before getting into the proof of our main results, we discuss an application of \autoref{thm:entropydissipation} for weak solutions to the spatially homogeneous Boltzmann equation. Let us have a look at the weak formulation of the Boltzmann collision operator. 
 In the following, let $f\in L^\infty([0,\infty); L_2^1(\R^d)) \cap C([0,\infty); \mathcal{D}'(\R^d))$ be a nonnegative function satisfying \eqref{ass:mass}, \eqref{ass:energy} and \eqref{ass:entropy}.

In \cite{Villaninewclass}, Villani introduces a class of weak solutions, called H-solutions, to the spatially homogeneous Boltzmann equation with bounded entropy dissipation. A solution in this class might not be a weak solution in the usual sense. As explained in \cite[Section 7, Application 2]{AlexDesvVilWenn} and \cite{Villaninewclass}, this problem appears because of the lack of an a priori estimate in the very soft potential case of the form
\begin{equation}\label{eq:apropri}
 \int_{0}^T \int_{B_R} \int_{B_R} f(t,v) f(t,v_{\ast})|v-v_{\ast}|^{\gamma+2} \, \dd v \, \dd v_{\ast} \, \dd t < \infty. 
\end{equation}
It is mentioned in \cite{AlexDesvVilWenn} that the (local) entropy dissipation estimate shows that H-solutions are weak solutions in the usual sense when $\gamma+2s \geq d-2$. The computation is sketched without explicit details. Using our estimate in Theorem \ref{thm:entropydissipation}, we show that \eqref{eq:apropri} holds whenever $\gamma+2s > -2$, covering the whole physical range of exponents. There is no fundamental difference between the computation presented here and the one proposed in \cite{DesvillettesColoumb}. It is not clear why they stated a suboptimal range in \cite{AlexDesvVilWenn}, suggesting that it is possibly a typo in the paper. We explain the computation explicitly below.

One important ingredient in the proof of \eqref{eq:apropri} is that H-solutions are in a weighted Lebesgue space. It follows immediately from \autoref{thm:entropydissipation} combined with the entropy dissipation formula \eqref{eqn:dissipationintegratedintime} (without integrating in space).

\begin{corollary}\label{corollary:weightedlebesgue}
Let $T>0$, $-d<\gamma\leq 2$ and $s\in(0,1)$. 
Let $f$ be a non-negative H-solution to the Cauchy problem \eqref{eq:homogboltzmann} with initial datum $f_0$. Assume $f_0$ satisfies \eqref{ass:entropy}. \\
Then $f\in L^{1}\left([0,T], L^p_{-q}(\R^d)\right)$, where $1/p = 1-2s/d$ and $q=2s/d - \gamma-2s$.
\end{corollary}
\begin{proof}
By definition, H-solutions satisfy the space-homogeneous form of \eqref{eqn:dissipationintegratedintime}. We get that $\int_0^T D(t) \dd t \leq \int f_0 \log f_0 \dd v$. The corollary follows applying Theorem \ref{thm:entropydissipation}.
\end{proof}

We use \autoref{corollary:weightedlebesgue} to show that H-solutions are weak solutions in the usual sense by proving \eqref{eq:apropri}.

\begin{corollary}\label{cor:aprioriestimate}
Let $T>0$, $-d<\gamma\leq 0$ and $s\in(0,1)$, so that $\gamma+2s > -2$. 
Let $f$ be a non-negative H-solution to the Cauchy problem \eqref{eq:homogboltzmann} with initial datum $f_0$. Assume $f_0$ satisfies \eqref{ass:entropy}. Then $f$ satisfies \eqref{eq:apropri}.
\end{corollary}

\begin{proof}
\begin{align*}
 \int_{0}^T & \int_{B_R}  \int_{B_R} f(t,v) f(t,v_{\ast}) |v-v_{\ast}|^{\gamma+2}\mathds{1}_{\{|v-v_{\ast}| \leq 1\}}) \, \dd v \, \dd v_{\ast} \, \dd t \\
& \leq \int_{0}^T \|f\|_{L^1} \sup_{v \in B_R} \int_{B_R} f(t,v_{\ast}) |v-v_{\ast}|^{\gamma+2} \, \dd v_{\ast} \, \dd t. \\
\intertext{We apply H\"older's inequality with $1/p = 1 - 2s/d$ as in Corollary \ref{corollary:weightedlebesgue}}
& \leq \int_{0}^T \|f(t,\cdot)\|_{L^1} \|f(t,\cdot)\|_{L^p(B_R)} \sup_{v \in B_R} \||v-\cdot|^{\gamma+2} \|_{L_{p'}(B_R)} \dd t \\
&\leq \|f\|_{L^\infty([0,T],L^1(B_R))} \|f\|_{L^1([0,T],L^p(B_R))} \sup_{v \in B_R} \||v-\cdot|^{\gamma+2} \|_{L_{p'}(B_R)}.
\end{align*}
It only remains to check whether the last factor is finite. We have
\begin{align*}
\||v-\cdot|^{\gamma+2} \|_{L_{p'}(B_R)} =\left( \int_{B_R} |v-v_\ast|^{(\gamma+2)\frac{d}{2s}} \, \dd v_\ast \right)^{\frac{2s}d}. 
\end{align*}
The integral is finite provided that $(\gamma+2)d/(2s) > -d$. This is clearly the case when $\gamma+2s > -2$.
\end{proof}

\section{Entropy dissipation estimates}
The Boltzmann collision operator clearly plays a central role in our analysis, since the entropy dissipation is defined through it. In the following, we will briefly discuss the decomposition of the operator and some selected properties. 

The Boltzmann collision operator $Q(f,g)$ can be decomposed into the sum of an integro-differential operator $Q_1(f,g)$ and a lower order term $Q_2(f,g)$ (see \cite{SilvestreBoltzmann}):
\begin{equation}\label{def:Q1Q2}
Q_1(f,g)(v) := (\mathcal{L}_{K_f}g)(v) \quad \text{and} \quad Q_2(f,g)(v) = g(v)\int_{\R^d}f(v-w)\widetilde{B}(|w|)\, \dd w.
\end{equation}
The integro-differential operator $\mathcal{L}_{K_f}$ is defined by
\[ \mathcal{L}_{K_f}g(t,v) = pv\int_{\R^d} (g(t,v')-g(t,v))K_{f}(t,v,v')\, \dd v',\]
where pv denotes the Cauchy principal value around $v\in\R^d$. The kernel $K_f(t,v,v')$ depends on the 
function $f$ and is given by the formula
\begin{equation}\label{def:kernel}
\begin{aligned}
 K_f(t,v,v') &= \frac{2^{d-1}}{|v'-v|}\int_{w\perp (v'-v)} f(t,v+w) \, B(r,\cos\Theta) r^{-d+2}\, \dd w,
 \end{aligned}
 \end{equation}
where
 \begin{equation}\label{def:rvwreps}
\begin{aligned}
r&=\sqrt{|v'-v|^2+|w|^2}, \qquad & \cos(\Theta/2)= \frac{|w|}{r}, \\
v_{\ast}'&= v+w, & v_{\ast}= v'+w.
\end{aligned}
\end{equation}
While $Q_1$ represents the singular part of the collision operator, the part $Q_2$ is of lower order. The lower order term can be handled using the cancellation lemma \cite[Lemma 1]{AlexDesvVilWenn}.
The function $\widetilde{B}$ in \eqref{def:Q1Q2} is given by
\begin{align*}
\widetilde{B}(z)& =\int_{\Sp^{d-1}}(2^{d/2}(1-\sigma\cdot e)^{-d/2}\left(B\left(\sqrt{2} z/(1-\sigma\cdot e),\cos\Theta\right) -B(z,\cos\Theta)\right) \, \dd \sigma\\
& =C_b \Phi(z)=C_bc_{\Phi}|z|^{\gamma},
\end{align*} 
where $C_b$ is a positive constant depending on the angular function $b$. For details on the decomposition of the Boltzmann collision operator, see \cite{SilvestreBoltzmann}.

The idea for controlling the singularity of $\tilde B$ near $v=v_{\ast}$ for $\gamma<0$ is to introduce an auxiliary collision kernel, where we replace the kinetic factor $\Phi$ by a smaller bounded function $\psi$ in which we cut out this singularity. 
Since the integrand of the entropy dissipation is non-negative, it can be bounded from below by the same expression with $B$ replaced by a smaller collision kernel.
For a given function $\psi$ on $\R^d$, we define the generalized collision kernel $B^{\psi}$ by
\[ B^{\psi}(|v-v_{\ast}|,\cos\Theta) = \psi(|v-v_{\ast}|)b(\cos\Theta). \]
The auxiliary collision kernel $B^{\psi}$ and the collision kernel $B$ differ only in the fact that we replace the kinetic factor $\Phi$ with the function $\psi$. 

By simply replacing $\Phi$ by $\psi$,  we can define the generalized kernel $K_f^{\psi}$ by
\begin{equation}\label{def:genkernel}
 \mathcal{K}_f^{\psi}(t,v,v') = \frac{2^{d-1}}{|v'-v|}\int_{w\perp (v'-v)} f(t,v'+w) \, \psi(r) b(\cos\Theta) r^{-d+2 }\, \dd w,
 \end{equation}
where $r$ and $\cos\Theta$ are defined in \eqref{def:rvwreps} and $\widetilde{B}^{\psi}$ by
\begin{equation}\label{def:genB}
\widetilde{B}^{\psi}(z) = C_b \psi(z).
 \end{equation}
This leads to the decomposition of the auxiliary collision operator $Q^{\psi}(f,g)$ (with the collision kernel $B$ replaced by $B^{\psi}$) into the sum of an integro-differential operator $Q^{\psi}_1(f,g)$ and a lower order term $Q^{\psi}_2(f,g)$, where the operators $Q_1^\psi$ and $Q_2^\psi$ 
are defined as in \eqref{def:Q1Q2} with $K_f$ and $\widetilde{B}$ replaced by $K_f^{\psi}$ resp. $\widetilde{B}^{\psi}$.

We define the auxiliary function $\psi:\R^d\to\R$ to be non-negative function satisfying $\psi\leq \Phi$ and:
\begin{equation}\label{def:psi}
\begin{aligned}
\text{if } \gamma<0:  \quad & \begin{cases}1\leq \psi(|z|) \leq 2 \quad &\text{if } |z|\leq 1, \\ \psi(|z|)=\Phi(|z|) &  \text{if } |z|> 1,
\end{cases} \\
\text{if } \gamma \geq  0:  \quad & \psi(|z|) =\Phi(|z|) \quad \text{for all } z\in\R^d.
\end{aligned}
\end{equation}
Note that by this choice, for any value of $\gamma$,
\[ \int_{\R^d} \int_{\R^d} f(v) f(v_{\ast}) \psi(|v-v_{\ast}|) \, \dd v_{\ast}\, \dd v \lesssim \begin{cases} 
M_0^2 & \text{ if } \gamma \leq 0, \\
\left( \int f \langle v \rangle^\gamma \dd v \right)^2 & \text{ if } \gamma > 0.
\end{cases} \]
The entropy dissipation is naturally connected to a quadratic form, coming from the singular part of the collision operator and a lower order term.
\begin{lemma}\label{lem:entropuquadratic}
Let $-d<\gamma\leq 2$ and $s\in(0,1)$. Let $\psi$ be a non-negative function satisfying $\psi\leq \Phi$ and \eqref{def:psi} and let $f$ be a non-negative function satisfying \eqref{ass:mass}, \eqref{ass:energy} and \eqref{ass:entropy}. 
There is an universal constant $C>0$ such that
\begin{equation}
D(f) \geq \int_{\R^d}\int_{\R^d} \left(\sqrt{f(v')} - \sqrt{f(v)}\right)^2 K^{\psi}_f(v,v') \, \dd v' \, \dd v  - C \left( \int_{\R^d} f(v) \langle v \rangle^{\gamma_+} \dd v \right)^2.
\end{equation}
\end{lemma}
\begin{proof} 
Using the non-negativity of the entropy dissipation $D(f)$ and $\psi\leq \Phi$, we get
\begin{align*}
D(f) & = \frac 12 \int_{\R^d}\int_{\R^d}\int_{\Sp^{d-1}} (ff_\ast - f' f'_\ast) \ln\left(\frac {ff_\ast}{f'f'_\ast} \right) B(|v-v_{\ast}|,\cos\Theta)\dd\sigma \dd v_{\ast} \dd v \\
& \geq \frac 12 \int_{\R^d}\int_{\R^d}\int_{\Sp^{d-1}} (ff_\ast - f' f'_\ast) \ln\left(\frac {ff_\ast}{f'f'_\ast} \right) \psi(|v-v_{\ast}|) b(\cos\Theta)\dd\sigma \dd v_{\ast} \dd v \\
& = \int_{\R^d}\int_{\R^d}\int_{\Sp^{d-1}} f(v_{\ast})f(v) \left[\ln(f(v))-\ln(f(v'))\right] \psi(|v-v_{\ast}|) b(\cos\Theta)\dd\sigma \dd v_{\ast} \dd v \\
& \geq \int_{\R^d}\int_{\R^d}\int_{\Sp^{d-1}} f(v_{\ast}) \left(\sqrt{f(v')} - \sqrt{f(v)}\right)^2\psi(|v-v_{\ast}|) b(\cos\Theta)\dd\sigma \dd v_{\ast} \dd v \\
& \quad -\int_{\R^d}\int_{\R^d}\int_{\Sp^{d-1}} f(v_{\ast}) \left(f(v') - f(v)\right)\psi(|v-v_{\ast}|) b(\cos\Theta)\dd\sigma \dd v_{\ast} \dd v \\
& = I_1 - I_2.
\end{align*}
In the second estimate, we used the inequality $x(\ln x - \ln y) \geq \left(\sqrt{y} - \sqrt{x}\right)^2 - (y-x)$ for all $x,y\geq 0$ (See \cite{AlexDesvVilWenn}).
By the definition of the kernel $K_f^{\psi}(v,v')$, the term $I_1$ can be written as
\[ I_1 = \int_{\R^d}\int_{\R^d} \left(\sqrt{f(v')} - \sqrt{f(v)}\right)^2 K_f^{\psi}(v,v') \, \dd v' \, \dd v. \]
For the term $I_2$  we use the cancellation lemma (See \cite{AlexDesvVilWenn} or \cite{SilvestreBoltzmann} for details), which leads to
\begin{align*}
I_2  &= C \int_{\R^d}\int_{\R^d} f(v)f(v_{\ast}) \psi(|v-v_{\ast}|) \dd v_{\ast} \dd v \\ 
&\leq C \int_{\R^d}\int_{\R^d} f(v)f(v_{\ast}) \langle v-v_{\ast} \rangle^\gamma \dd v_{\ast} \dd v \\ 
 &\leq C \begin{cases} 
M_0^2 & \text{ if } \gamma \leq 0, \\
\left( \int f \langle v \rangle^\gamma \dd v \right)^2 & \text{ if } \gamma > 0.
\end{cases}
\end{align*}
Here $C>0$ is a finite universal constant.
\end{proof}

Our main result \autoref{thm:entropydissipation} will be derived from \autoref{lem:entropuquadratic} and the following estimate of the quadratic form from below by a weighted Lebesgue norm.

\begin{proposition}\label{prop:quadratictoLebesgue}
Let $\gamma>-d$ and $s\in(0,1)$. Let $\psi$ be a non-negative function satisfying \eqref{def:psi}, let $f$ be a non-negative function satisfying \eqref{ass:mass}, \eqref{ass:energy} and \eqref{ass:entropy}. Let $p$, $q$ and $r$ be the exponents given by $1/p = 1-2s/d$, $q=2s/d - \gamma-2s$, and $r = -\gamma-d+1$. Assume $g \in L^p_{-q}(\R^d)$ and let $a = C_1 \|g\|_{L^p_{-q}}$, for some large constant $C_1$. There is a constant $c_0>0$, such that
\[  \int_{\R^d}\int_{\R^d} \left(\sqrt{g(v')} - \sqrt{g(v)}\right)^2 K^{\psi}_f(v,v') \, \dd v' \, \dd v \geq c_0 \|g\|_{L^p_{-q}}^{1-p} \int_{\{g(v) \geq a \langle v \rangle^r \}} |g(v)|^p \langle v \rangle^{-qp} \dd v. \]
Moreover, when $q \leq 0$ we can take $C_1=0$ and the right hand side is simply $\|g\|_{L^p_{-q}}$.

Here, the constants $C_1$ and $c_0$ depend only on the dimension $d$, $s$, $\gamma$ and the macroscopic bounds $m_0,M_0,E_0$ and $H_0$
\end{proposition}
The proof of Proposition \ref{prop:quadratictoLebesgue} is postponed to later in this section, after the next four lemmas.

Since the estimate in \autoref{prop:quadratictoLebesgue} has no restrictions on $\gamma>-d$ and $s\in(0,1)$, 
it covers soft as well as hard potentials for the Boltzmann collision operator. It is perhaps most interesting that it works in the case of very-soft potentials $\gamma+2s<0$. Note that outside of that range, if $\gamma+2s>2s/d$, the exponent $q$ changes its sign. We have $q>0$ in the very soft potential range.

An essential tool for the proof of \autoref{prop:quadratictoLebesgue} are cones of nondegeneracy introduced in \cite{SilvestreBoltzmann}.
Before we recall the cones of nondegeneracy and some important properties, we first give a lower bound on the generalized kernel.
\begin{lemma}\label{lemma:kernelpsi}
Let $-d<\gamma<0$ and $s\in(0,1)$. Let $\psi$ be a non-negative function satisfying \eqref{def:psi}.
Then 
\begin{equation}\label{Kernelpsi}
\begin{aligned}
K^{\psi}_f(v,v') \gtrsim |v-v'|^{-d-2s}\int_{w\perp (v'-v)} f(v'+w)  \min\Big( |w|^{\gamma+2s+1}, |w|^{2s+1}\Big) \, \dd w.
 \end{aligned}
\end{equation}
\end{lemma}
Note that in the hard potential case $\gamma\geq 0$, the auxiliary kernel $K_f^{\psi}$ coincides with the kernel $K_f$ and therefore, there is nothing to do in this case. The respective result is given in \cite[Corollary 4.2]{SilvestreBoltzmann}, namely
\begin{equation}\label{eq:lowerboundKf}
K_f(v,v') \gtrsim |v-v'|^{-d-2s}\int_{w\perp (v'-v)} f(v'+w)   |w|^{\gamma+2s+1} \, \dd w,
\end{equation} 
which provides a better lower bound. Nevertheless, the estimate \eqref{Kernelpsi} is sufficient for our applications in the case of soft potentials.
\begin{proof}[Proof of \autoref{lemma:kernelpsi}]
As in the proof of \cite[Corollary 4.2]{SilvestreBoltzmann}, we study the two  cases $\cos\Theta\geq 0$ and $\cos\Theta< 0$. 
\begin{enumerate}
\item[(i)] If $\cos\Theta\geq 0$, then $|w|\approx r$ and $b(\cos\Theta)\approx |v-v'|^{-d+1-2s}r^{d-1+2s}$. Hence, 
\begin{align*} 
\psi(r) b(\cos\Theta) r^{-d+2 } &\approx |v-v'|^{-d-2s+1}\left( |w|^{1+2s}\mathds{1}_{\{r\leq 1\}}(r) + |w|^{1+2s+\gamma}\mathds{1}_{\{r> 1\}}(r)\right) \\
&\geq  |v-v'|^{-d-2s+1}\min\Big( |w|^{\gamma+2s+1}, |w|^{2s+1}\Big).
\end{align*}
\item[(ii)] In the case $\cos\Theta<0$, we have $|v'-v|\approx r$ and $|w|=r\cos(\Theta/2)$. 
Therefore, $b(\cos\Theta) = \cos(\Theta/2)^{\gamma+2s+1}$. If $r\leq 1$, then 
\begin{align*}
\psi(r) b(\cos\Theta) r^{-d+2} &\approx r^{-d-2s+1} |w|^{\gamma+2s+1} \\
&\approx  |v'-v|^{-d-2s+1} |w|^{\gamma+2s+1} \\
&\geq  |v-v'|^{-d-2s+1}\min\Big( |w|^{\gamma+2s+1}, |w|^{2s+1}\Big).
\end{align*}
On the other hand, if $r\geq 1$, we have
\begin{align*}
\psi(r) b(\cos\Theta) r^{-d+2} &\approx r^{-d+2} \cos(\Theta/2)^{\gamma+2s+1}r^{\gamma} \\
&\approx |v-v'|^{-d-2s+1}|w|^{2s+1}\cos(\Theta/2)^{\gamma}r^{\gamma}  \\
&\approx |v-v'|^{-d-2s+1}|w|^{\gamma+2s+1} \\
& \geq |v-v'|^{-d-2s+1}\min\Big( |w|^{\gamma+2s+1}, |w|^{2s+1}\Big).
\end{align*}
\end{enumerate}
This finishes the proof of \autoref{lemma:kernelpsi}.
\end{proof}
Note that by the bound on the mass and energy, a certain amount of the mass of the function $f$ lies inside a ball centered around zero with radius depending on $m_0$ and $E_0$. The bound on the entropy $H_0$ provides that this mass in not concentrated on a set of measure zero.
These observations lead to cones of nondegeneracy constructed by sets of the form $\{f\geq \ell\}$. To be more precise, for any point $v\in\R^d$, there is a symmetric cone of directions $A(v)$ such that its perpendicular planes intersect the set $\{f\geq \ell\}$ on a set with $\mathcal{H}^{d-1}$ positive Hausdorff measure. 
As a consequence of \autoref{lemma:kernelphi} resp. \eqref{eq:lowerboundKf}, 
we get the existence of a cone of non-degeneracy for the kernel $K^{\psi}_f$. 
Here, one can simply follow the lines of the proof of \cite[Lemma 4.8 and Lemma 7.1]{SilvestreBoltzmann} and use the bound on the kernel $K_f^{\psi}$ given in Lemma \ref{lemma:kernelpsi}
\begin{lemma}{\cite[Lemma 7.1]{SilvestreBoltzmann}}\label{lemma:kernelphi}
Let $\gamma>-d$ and $s\in(0,1)$.  Let $\psi$ be a non-negative function satisfying \eqref{def:psi}
and let $f$ be a non-negative function satisfying \eqref{ass:mass}, \eqref{ass:energy} and \eqref{ass:entropy}.
For any $v\in\R^d$, there exists a symmetric subset $A(v)\subset\Sp^{d-1}$ such that 
\begin{enumerate}
\item[(i)] $\displaystyle |A(v)|>\mu\langle v \rangle^{-1}$, where $|A(v)|$ denotes the $d-1$-Hausdorff measure of $A(v)$,
\item[(ii)]  $\displaystyle  K^{\psi}_f(v,v')\geq \lambda \langle v \rangle^{1+2s+\gamma}|v-v'|^{-d-2s}$ whenever $(v'-v)/(|v'-v|)\in A(v)$.
\item[(iii)] For every $\sigma\in A(v)$, $|\sigma\cdot v| \leq C$.
\end{enumerate}
The constants $\mu$, $\lambda$ and $C$ depend on $d$ and on the hydrodynamic bounds $m_0, M_0, E_0$ and $H_0$.
\end{lemma}
The set $A(v)$ in the previous lemma describes a set of directions $A(v)$ along which the kernel $K^{\psi}_f$ has a lower bounds given in property (ii). We denote the corresponding cone of nondegeneracy by $\Xi(v)$ that is
\[ \Xi(v) := \left\{  v' \in\R^d\colon \frac{(v'-v)}{|v'-v|}\in A(v)\right\}. \]
Furthermore, the cone of nondegeneracy degenerates as $|v|\to\infty$ and satisfies
\[ |B_R(v) \cap \Xi(v)| \approx R^d \langle v \rangle^{-1} . \]

The proof of Proposition \ref{prop:quadratictoLebesgue} depends on estimating the size of the set of points in the cone of nondegeneracy so that $g(v') < g(v)/2$. The computation rather straight forward when $q \leq 0$, and slightly more involved when $q >0$. Let us start with a lemma for the easier case.

\begin{lemma} \label{l:qpositive}
Let $p$, $q$ be exponents as in Proposition \ref{prop:quadratictoLebesgue}. Assume $q \leq 0$. For a universal constant $C_1$ large enough, let $a := C_1 \|g\|_{L^p_{-q}}$. For any $v \in \R^d$, choose $R$ so that $g(v)^p R^d \langle v \rangle^{-qp-1} = a^p$. Then
\[ \left\vert\left\{ v' \in B_R(v) \cap \Xi(v) : g(v') \geq g(v)/2 \right\}\right\vert \geq c R^d \langle v \rangle^{-1}, \]
for some universal constant $c>0$.
\end{lemma}

\begin{proof}
If $R < |v|/2$, we observe that $\langle v' \rangle \approx \langle v \rangle$ for all $v' \in B_R(v)$. We use Chebyshev's inequality and get
\begin{align*}
\left\vert\left\{ v' \in B_R(v) \cap \Xi(v) : g(v') \geq g(v)/2 \right\}\right\vert &\lesssim \langle v \rangle^{pq}  g(v)^{-p} \int_{\left\{ v' \in B_R(v) \cap \Xi(v) : g(v') \geq g(v)/2 \right\}} g(v')^p \langle v' \rangle^{-pq} \dd v' \\
&\lesssim \langle v \rangle^{pq}  g(v)^{-p} \|g\|_{L^p_{-q}}^p.
\end{align*}

Our estimates on the cone of nondegeneracy say that $|B_R(v) \cap \Xi(v)| \approx R^d \langle v \rangle^{-1} = a^p g(v)^{-p} \langle v \rangle^{qp}$. Thus, we can make sure that $g(v') \geq g(v)/2$ holds for less than half of the points in $B_R(v) \cap \Xi(v)$ (in measure) by choosing $C_1$ large enough. 

If $R > |v|/2$, we repeat the same argument but replacing $B_R(v)$ with $B_{8R}(v) \setminus B_{4R}(v)$.
\end{proof}

The case $q > 0$ is the most interesting. In this case we get a version of Lemma \ref{l:qpositive} that applies only for those points $v$ so that $g(v)$ is large enough.

\begin{lemma} \label{l:goodset}
Let $p$, $q$ and $r$ be exponents as in Proposition \ref{prop:quadratictoLebesgue}. Assume $q > 0$. For a large enough constant $C_1$, let $a := C_1 \|g\|_{L^p_{-q}}$. For any $v \in \R^d$ such that $g(v) \geq a \langle v \rangle^{r}$, choose $R$ so that $g(v)^p R^d \langle v \rangle^{-qp-1} = a^p$. Then
\[ \left\vert\left\{ v' \in B_R(v) \cap \Xi(v) : g(v') \geq g(v)/2 \right\}\right\vert \geq c R^d \langle v \rangle^{-1}, \]
for some constant $c>0$.

Here, the constants $C_1$ and $c$ depend only on the dimension $d$, $s$, $\gamma$ and the macroscopic bounds $m_0,M_0,E_0$ and $H_0$
\end{lemma}

\begin{proof}
Since $g(v) \geq a\langle v \rangle^r$, we see that
\begin{align*}
 R &:= \left( a^p g(v)^{-p} \langle v \rangle^{qp+1}\right)^{1/d} \\
 &\leq \left( \langle v \rangle^{qp+1-rp}\right)^{1/d} = \langle v \rangle.
\end{align*}
The last equality holds because of our choice of $r = -d -\gamma + 1 =  q - (d-1)/p$.

Since $R \leq \langle v \rangle$, we have that $\langle v' \rangle \lesssim \langle v \rangle$ for $v' \in B_R(v)$. Therefore, if $q \geq 0$, using Chebyshev's inequality,
\begin{align*}
\left\vert\left\{ v' \in B_R(v) \cap \Xi(v) : g(v') \geq g(v)/2 \right\}\right\vert &\lesssim \langle v \rangle^{pq}  g(v)^{-p} \int_{\left\{ v' \in B_R(v) \cap \Xi(v) : g(v') \geq g(v)/2 \right\}} f(v')^p \langle v' \rangle^{-pq} \dd v' \\
&\lesssim \langle v \rangle^{pq}  g(v)^{-p} \|g\|_{L^p_{-q}}^p
\end{align*}

Our estimates on the cone of nondegeneracy say that $|B_R(v) \cap \Xi(v)| \approx R^d \langle v \rangle^{-1} = a^p g(v)^{-p} \langle v \rangle^{qp}$. Thus, we can make sure that $g(v') \geq g(v)/2$ holds for less than half of the points in $B_R(v) \cap \Xi(v)$ (in measure) by choosing $C_1$ large enough.

\end{proof}

We can finally prove \autoref{prop:quadratictoLebesgue}.
\begin{proof}[Proof of \autoref{prop:quadratictoLebesgue}]
We describe the proof in the more interesting case $q > 0$. The case $q \leq 0$ follows the same steps applying Lemma \ref{l:qpositive} instead of Lemma \ref{l:goodset}.


Recall that by \autoref{lemma:kernelphi} for every $v\in\R^d$, we know that $K^{\psi}_f(v,v') \geq \lambda \langle v \rangle^{1+2s+\gamma}|v-v'|^{-d-2s}$, whenever $v'\in \Xi(v)$. 
Furthermore, if $g(v) \geq a \langle v \rangle^r$, let us choose $R=R(v)>0$ like in Lemma \ref{l:goodset}, so that for some large $C_1>0$,
\[ g(v)^p R^d \langle v\rangle^{-qp-1} = a^p = C_1^p \|g\|_{L^p_{-q}}^p. \]
With this choice, applying Lemma \ref{l:goodset},
\[|\{v'\in (B_R(v)\cap \Xi(v))\colon g(v') \leq g(v)/2\}| \gtrsim  |B_R(v)\cap \Xi(v)| \approx R^d \langle v \rangle^{-1}.	\]
Hence, using this information, we get
\begin{align*}
\int_{\R^d} \left( \sqrt{g(v')} - \sqrt{g(v)} \right)^2K_f^{\psi}(v,v')\, \dd v'  & \gtrsim R^d \langle v \rangle^{-1} g(v) \left(\lambda \langle v\rangle^{1+2s+\gamma}R^{-d-2s} \right) \\
& \geq c_1 R^{-2s} g(v) \langle v \rangle^{\gamma+2s}  \\
& = c_2 \|g\|_{L^p_{-q}}^{-2sp/d} g(v)^{1+2sp/d}\langle v \rangle^{\gamma+2s-2s(qp+1)/d},
\end{align*}
where $c_1,c_2$ are constants depending on $m_0, M_0, E_0, H_0$.
Our choice of $p$ and $q$ was made so that $p=1+2sp/d$ and $-qp =\gamma+2s-2s(qp+1)/d$. 
Integrating over all those $v\in\R^d$ so that $g(v) \geq a\langle v \rangle^r$ finally gives us
\begin{align*} \int_{\R^d}\int_{\R^d} \left( \sqrt{g(v')} - \sqrt{g(v)} \right)^2K_f^{\psi}(v,v')\, \dd v' \dd v   &\geq c\|g\|_{L^p_{-q}}^{-2sp/d} 
\int_{g(v) \geq a \langle v \rangle^r} g(v)^p \langle v \rangle^{-qp} \dd v. 
\end{align*}
Since $-2sp/d = 1-p$, we conclude the proof. 
\end{proof}

When $q>0$, the estimate from Proposition \ref{prop:quadratictoLebesgue} needs to be improved to account for those points $v$ so that $g(v) \leq a \langle v \rangle^r$. This is the purpose of the next Lemma.

\begin{lemma} \label{lem:smallvalues}
Let $p$, $q$ and $r$ be as in Proposition \ref{prop:quadratictoLebesgue}. Assume $q > 0$ and $g \in L^p_{-q}(\R^d) \cap L^1(\R^d)$. Then
\[ \int_{\{g(v) < a \langle v \rangle^r \}} |g(v)|^p \langle v \rangle^{-qp} \dd v \leq a^{p-1} \int_{\R^d} |g(v)| \dd v.\]
\end{lemma}

\begin{proof}
We simply bound the integrand $|g(v)|^p < a^{p-1} \langle v \rangle^{(p-1)r} g(v)$ for every $v$ and proceed
\begin{align*}
\int_{\{g(v) < a \langle v \rangle^r \}} |g(v)|^p \langle v \rangle^{-qp} \dd v &\leq a^{p-1} \int_{\R^d} |g(v)| \langle v \rangle^{-qp + (p-1)r} \dd v.
\end{align*}
We finish the proof by observing that $-qp + (p-1)r = p \gamma (1-2s/d) < 0$ with our choice of exponents when $q > 0$ (and in particular $\gamma<0$).
\end{proof}


Using \autoref{lem:entropuquadratic}, \autoref{prop:quadratictoLebesgue} and \autoref{lem:smallvalues}, we derive \autoref{thm:entropydissipation}.

\begin{proof}[Proof of \autoref{thm:entropydissipation}]
Let us start by assuming that $f \in L^p_{-q}(\R^d)$ and deduce the a priori estimate. From \autoref{lem:entropuquadratic}, we have that
\[ D(f) \geq \int_{\R^d}\int_{\R^d} \left(\sqrt{f(v')} - \sqrt{f(v)}\right)^2 K^{\psi}_f(v,v') \, \dd v' \, \dd v  - C \left( \int_{\R^d} f(v) \langle v \rangle^{\gamma_+} \dd v \right)^2. \]
Moreover, applying \autoref{prop:quadratictoLebesgue}, we estimate the double integral as
\begin{align*}
\int_{\R^d}\int_{\R^d} & \left(\sqrt{f(v')} - \sqrt{f(v)}\right)^2 K^{\psi}_f(v,v') \, \dd v' \, \dd v \geq c_0 \|f\|_{L^p_{-q}}^{1-p} \int_{\{f(v) \geq a \langle v \rangle^r \}} |f(v)|^p \langle v \rangle^{-qp} \dd v \\
&\geq c_0 \left( \|f\|_{L^p_{-q}} - \|f\|_{L^p_{-q}}^{1-p} \int_{\{f(v) \leq a \langle v \rangle^r \}} |f(v)|^p \langle v \rangle^{-qp} \dd v \right).
\end{align*}

Combining these two inequalities with \autoref{lem:smallvalues}, we are left with
\[ D(f) \geq c_0 \|f\|_{L^p_{-q}} - c_0 C_1^{p-1} M_0 - C \left( \int_{\R^d} f(v) \langle v \rangle^{\gamma_+} \dd v \right)^2. \]
The constants $c_0$ and $C_1$ are chosen sufficiently small and sufficiently large respectively in Proposition \ref{prop:quadratictoLebesgue}. We may choose $c_0$ smaller if necessary so that $c_0 C_1^{p-1} < M_0$. Thus,
\[ D(f) \geq c_0 \|f\|_{L^p_{-q}} - M_0^2 - C \left( \int_{\R^d} f(v) \langle v \rangle^{\gamma_+} \dd v \right)^2. \]
This concludes the proof of \autoref{thm:entropydissipation} when $f \in L^p_{-q}(\R^d)$.

For a function $f \notin L^p_{-q}$, we consider $f_m(v) := \min(f(v),m)$. Since $f \in L^1_2(\R^d)$, we have $f_m \in L^p_{-q}$ and \autoref{prop:quadratictoLebesgue} holds for $g=f_m$. Moreover, $\|f_m\|_{L^p_{-q}} \to \infty$ and therefore, applying the monotone convergence theorem,
\begin{align*}
\int_{\R^d}\int_{\R^d} &\left( \sqrt{f(v')} - \sqrt{f(v)} \right)^2 K_f^{\psi}(v,v')\, \dd v' \dd v \geq \\
& \geq \lim_{m \to \infty} \int_{\R^d}\int_{\R^d} \left( \sqrt{f_m(v')} - \sqrt{f_m(v)} \right)^2K_f^{\psi}(v,v')\, \dd v' \dd v \\
& \geq \lim_{m \to \infty} c_0 \|f_m\|_{L^p_{-q}}^{1-p} \int_{\{f(v) \geq a \langle v \rangle^r \}} |f_m(v)|^p \langle v \rangle^{-qp} \dd v = +\infty. 
\end{align*}
In view of \autoref{lem:entropuquadratic}, we must have $D(f)=+\infty$.
\end{proof}

\subsection{Entropy dissipation estimate involving an anisotropic fractional Sobolev space}
In this subsection, we present a second entropy dissipation estimate for the Boltzmann collision operator. 
This estimate involves the anisotropic distance by Gressmann and Strain \cite{Gressmann-Strain-2011}.

Let us first briefly recall the sharp anisotropic coercivity estimate for the Boltzmann collision operator from \cite{Gressmann-Strain-2011}.
Note that $\left\langle Q(g,f),f\right\rangle$ can be rewritten as
\begin{equation}\label{eq:NgKg}
\begin{aligned}
\left\langle Q(g,f),f\right\rangle &= \int_{\R^d}\int_{\R^d}\int_{\Sp^{d-1}} g(v_{\ast})f(v) \left[f(v')-f(v)\right] B(r,\cos\Theta)\dd\sigma \dd v_{\ast} \dd v \\
& = \frac12 \int_{\R^d}\int_{\R^d}\int_{\Sp^{d-1}} g(v_{\ast})\left[f(v')^2 -f(v)^2\right] B(r,\cos\Theta)\dd\sigma \dd v_{\ast} \dd v \\
& \quad \qquad -\frac12 \int_{\R^d}\int_{\R^d}\int_{\Sp^{d-1}} g(v_{\ast})\left[f(v')-f(v)\right]^2 B(r,\cos\Theta)\dd\sigma \dd v_{\ast} \dd v\\
& =:K_g(f) - N_g(f).
\end{aligned}
\end{equation}

In \cite[Theorem 1]{Gressmann-Strain-2011}, the authors prove that under mild assumptions on the function $g$, the term $N_g(f)$ can be bounded from below the weighted anisotropic Sobolev semi-norm
$|f|_{{\dot{N}^{s,\gamma}}}^2$ defined in \eqref{def:weightedanisoSobolev}. This estimate provides an entropy dissipation estimate in terms of the anisotropic fractional Sobolev space. However, the result depends on certain parameter $C_q$ that would be difficult to bound when $\gamma<0$. In this section we prove Proposition \ref{prop:entropydissipation}, which is effectively a refinement of \cite[Theorem 1]{Gressmann-Strain-2011} in the soft potential range.

In \cite{imbertsilvestreglobal2022}, Imbert and Silvestre introduce a change of variables, which is used to turn local Hölder and Schauder estimates into global ones. For $v_0\in\R^d$ let 
\[ \overline{v} :=  \begin{cases} v_0 + v  \quad & \text{if } |v_0|<2, \\ v_0+T_0v &  \text{if } |v_0|\geq 2,\end{cases} \]
 where
 \begin{equation}\label{def:T0}
  T_0(av_0+w) = \frac{a}{|v_0|}v_0 + w, \qquad \text{for all } w\perp v_0, a\in\R. 
 \end{equation}
 The function $T_0:\R^d\to\R^d$, introduced in \cite{imbertsilvestreglobal2022}, has a strong connection to the anisotropic distance 
\eqref{eq:GSdist}.
In \cite[Lemma A.1]{imbertsilvestreglobal2022} it is shown that for any given $v_0\in\R^d$ with $|v_0|\geq 2$, we have 
\begin{equation}\label{eq:comparabilitymetric}
d_{GS}(v_1,v_2) \asymp |T_0^{-1}(v_1-v_2)|
\end{equation}
for all $v_1,v_2\in E_1(v_0) := v_0+T_0(B_1)$. 
We define
\begin{equation}\label{def:Kbar}
\overline{K}^{\psi}_f(v,v') = |v_0|^{-1-\gamma-2s}K^{\psi}_f(\overline{v}, v_0+T_0v'),
\end{equation}
where $\psi$ is the auxiliary function $\psi:\R^d\to\R$ defined in
\eqref{def:psi}. 
In \cite{imbertsilvestreglobal2022}, the authors derive the global coercivity estimate by Gressmann and Strain by using the above mentioned transformation and a local coercivity estimate. 
By using similar methods, we are able to prove \autoref{prop:entropydissipation}.
Before we draw our attention to the proof of \autoref{prop:entropydissipation}, we need some auxiliary results.
Let us first state a local coercivity estimate.
\begin{lemma}\label{lemma:localcoercivity}
Let $\gamma>-d$ and $s\in(0,1)$. Let $\psi$ be a non-negative function satisfying $\psi\leq \Phi$ and \eqref{def:psi} and let $f$ be a non-negative function satisfying \eqref{ass:mass}, \eqref{ass:energy} and \eqref{ass:entropy}. 
There is a constant $\lambda>0$, depending on the macroscopic bounds $m_0, M_0, E_0$ and $H_0$, such that for every $g:\R^d\to\R$
\[ \int_{B_1}\int_{B_1} \left(g(v') - g(v)\right)^2 \overline{K}^{\psi}_f(v,v') \dd v' \dd v \geq \lambda \int_{B_{1/2}}\int_{B_{1/2}} \frac{\left(g(v') - g(v)\right)^2}{|v-v'|^{d+2s}} \dd v' \dd v. \]
\end{lemma}
\begin{proof} 
By \autoref{lemma:kernelphi}, there is a cone of non-degeneracy for the kernel $K^{\psi}_f$. Hence, following the lines of the proof of \cite[Lemma 5.6]{imbertsilvestreglobal2022}, there is also a cone of non-degeneracy for $\overline{K}^{\psi}_f$. 
Now the result follows from the coercivity condition \cite[Theorem 1.3]{chakersilvestrecoerc}.
\end{proof}
Let $v_0\in\R^d\setminus B_2$ be given.
For $v\in\R^d$, let $\overline{v}$ be such that $v=v_0+T_0(\overline{v})$, where $T_0$ is defined in \eqref{def:T0}. Furthermore, let $\overline{g}(v)= g(\overline{v})$.
\begin{lemma}\label{lemma:coerc1}
Let $\gamma>-d$ and $s\in(0,1)$. Let $\psi$ be a non-negative function satisfying $\psi\leq \Phi$ and \eqref{def:psi} and let $f$ be a non-negative function satisfying \eqref{ass:mass}, \eqref{ass:energy} and \eqref{ass:entropy}. 
There are $c>0$, $R\in (2,\infty)$ and $\rho\in (0,1]$, depending on the macroscopic bounds $m_0, M_0, E_0$ and $H_0$, such that for all $g:\R^d\to\R$
\begin{align*} 
\iint_{d_{GS}(v,v')<R} &(g(v)- g(v'))^2 K_f^{\psi}(v,v')\,\, \dd v' \, \dd v   \\
& \geq c\iint_{d_{GS}(v,v')<\rho}  (g(v)- g(v'))^2\, \frac{\left( \langle v \rangle \langle v' \rangle \right)^{(\gamma+2s+1)/2}}{d_{GS}(v,v')^{d+2s}}\, \dd v' \, \dd v . 
\end{align*}
\end{lemma}
\begin{proof} The proof follows as in \cite[Lemma A.2]{imbertsilvestreglobal2022}. 
It uses the transformation \eqref{def:T0}, the comparability \eqref{eq:comparabilitymetric} and the local coercivity estimate \autoref{lemma:localcoercivity}.
Having these results on hand, one can proceed in the exact same way as it is done in \cite[Lemma A.2]{imbertsilvestreglobal2022}. 
\end{proof}

The integral of Lemma \ref{lemma:coerc1} is over the points $v,v' \in \R^d$ so that $d_{GS}(v,v') < \rho$. This value of $\rho$ is not important. The next lemma shows that we can enlarge it at will without altering the result.

\begin{lemma} \label{lem:rho32rho}
Let $g : \R^d \to \R$ be any measurable function, $q \in \R$ and $\rho < 2$. The following inequality holds for some $c>0$ depending only on dimension, $s$ and $q$.
\begin{align*}
\iint_{d_{GS}(v,v')<\rho}  (g(v)- g(v'))^2\, \frac{\left( \langle v \rangle \langle v' \rangle \right)^q}{d_{GS}(v,v')^{d+2s}}\, \dd v' \, \dd v & \\
\geq c \iint_{d_{GS}(v,v')< \frac 32 \rho}  (g(v)- g(v'))^2\, &\frac{\left( \langle v \rangle \langle v' \rangle \right)^q}{d_{GS}(v,v')^{d+2s}}\, \dd v' \, \dd v.
\end{align*}
\end{lemma}

\begin{proof}
Given $v,v' \in \R^d$ so that $d_{GS}(v,v') < \frac 32 \rho$, define
\[ N(v,v') := \left\{ w \in \R^d : d_{GS}(v,w) < \frac 23 d_{GS}(v,v') \text{ and } d_{GS}(v',w) < \frac 23 d_{GS}(v,v') \right\}.\]

From the triangle inequality, we observe that $d_{GS}(v,w) > \frac 13 d_{GS}(v,v')$ for all $w \in N(v,v')$. Moreover, since $d_{GS}(v,v') < \frac 32 \rho < 3$, we also see that $\langle v \rangle \approx \langle v' \rangle \approx \langle w \rangle$ for all $w \in N(v,v')$.

Let us also define $M(v,w) := \{ v' \in \R^d: w \in N(v,v')\}$. The sets $N(v,v')$ and $M(v,w)$ are substantial portions of a ball with respect to the distance $d_{GS}$. It is not hard to estimate their volumes $|N(v,v')| \approx |M(v,w)| \approx \langle v \rangle^{-1} d_{GS}(v,v')^d$.

With this notation, we proceed with the computation
\begin{align*}
\iint_{d_{GS}(v,v')< \frac 32 \rho}  &(g(v)- g(v'))^2\, \frac{\left( \langle v \rangle \langle v' \rangle \right)^q}{d_{GS}(v,v')^{d+2s}}\, \dd v' \, \dd v \\
&= \iint_{d_{GS}(v,v')< \frac 32 \rho}  (g(v)- g(v'))^2\, \frac{\left( \langle v \rangle \langle v' \rangle \right)^q}{d_{GS}(v,v')^{d+2s}}\, |N(v,v')|^{-1} \int_{N(v,v')} \dd w \dd v' \, \dd v. \\
\intertext{Using that $(g(v)- g(v'))^2 \leq 2(g(v)- g(w))^2 + 2(g(w)- g(v'))^2$ and $|N(v,v')| \approx \langle v \rangle^{-1} d_{GS}(v,v')^d$}
&\lesssim \iint_{d_{GS}(v,v')< \frac 32 \rho}\int_{N(v,v')}  \left( (g(v)- g(w))^2 + (g(w)- g(v'))^2 \right) \, \frac{ (\langle v \rangle \langle v' \rangle)^{q+1/2}}{d_{GS}(v,v')^{2d+2s}}\, \dd w \dd v' \, \dd v \\
\intertext{Note the symmetry respect to $v$ and $v'$ and that $\langle v \rangle \approx \langle v' \rangle \approx \langle w \rangle$.}
& \approx \iint_{d_{GS}(v,w)<\rho} (g(v)- g(w))^2 \frac{ \langle v \rangle^{2q+1}}{d_{GS}(v,w)^{2d+2s}} \left( \int_{v' \in M(v,w)} \dd v' \right) \dd w \dd v \\
& \approx \iint_{d_{GS}(v,w)<\rho} (g(v)- g(w))^2 \frac{ \langle v \rangle^{q} \langle w \rangle^{q}}{d_{GS}(v,w)^{d+2s}} \dd w \dd v.
\end{align*}
\end{proof}

\begin{corollary} \label{cor:rhoisirrelevant}
Let $g : \R^d \to \R$ be any measurable function, $q \in \R$ and $\rho < 1$. The following inequality holds for some $c>0$ depending only on dimension, $s$, $q$ and $\rho$.
\begin{align*}
\iint_{d_{GS}(v,v')<\rho}  (g(v)- g(v'))^2\, \frac{\left( \langle v \rangle \langle v' \rangle \right)^q}{d_{GS}(v,v')^{d+2s}}\, \dd v' \, \dd v & \\
\geq c \iint_{d_{GS}(v,v')< 1}  (g(v)- g(v'))^2\, &\frac{\left( \langle v \rangle \langle v' \rangle \right)^q}{d_{GS}(v,v')^{d+2s}}\, \dd v' \, \dd v.
\end{align*}
\end{corollary}

\begin{proof} We iterate Lemma \ref{lem:rho32rho} $m$ times so that $(3/2)^m \rho \geq 1$.
\end{proof}

We finally have all tools to proof our main result of this section, that is \autoref{prop:entropydissipation}.
\begin{proof}[Proof of \autoref{prop:entropydissipation}]
Let $\psi$ be a non-negative satisfying $\psi\leq \Phi$ and \eqref{def:psi}. 
Proceeding as in the proof of \autoref{lem:entropuquadratic},
\[ D(f) \geq  N_f^{\varphi}(\sqrt{f}) - C M_0^2,\]
where $C>0$ is a finite universal constant and
\[ N_f^{\psi}(\sqrt{f}) = \int_{\R^d}\int_{\R^d} \left(\sqrt{f(v')} - \sqrt{f(v)}\right)^2 K_f^{\psi}(v,v') \dd v' \dd v.\]

By \autoref{lemma:coerc1}, there are $c_1>0$ and $\rho\in(0,1)$, depending on $m_0, M_0, E_0$ and $H_0$, such that
\[ N_f^{\psi}(\sqrt{f}) \geq c_1 \iint_{d_{GS}(v,v')<\rho}  (\sqrt{f(v)}- \sqrt{f(v')})^2\, \frac{\left( \langle v \rangle \langle v' \rangle \right)^{(\gamma+2s+1)/2}}{d_{GS}(v,v')^{d+2s}}\, \dd v' \, \dd v.\]

Because of Corollary \ref{cor:rhoisirrelevant}, we can replace $\rho$ in the formula above by $1$ by adjusting the constant $c_1$. We get
\[ N_f^{\psi}(\sqrt{f}) \geq c_2 \iint_{d_{GS}(v,v')<1}  (\sqrt{f(v)}- \sqrt{f(v')})^2\, \frac{\left( \langle v \rangle \langle v' \rangle \right)^{(\gamma+2s+1)/2}}{d_{GS}(v,v')^{d+2s}}\, \dd v' \, \dd v.\]
Therefore, we conclude that
\[ D(f) \geq c |\sqrt{f}|_{\dot{N}^{s,\gamma}}^2 - C M_0^2. \]

\end{proof}

\bibliographystyle{abbrv}
\bibliography{lit}

\end{document}